\documentclass[11pt]{amsart}

\usepackage{amsmath, amsthm}
\usepackage{amssymb}
\usepackage{amsfonts}
\usepackage{latexsym}
\usepackage{mathrsfs}
\usepackage{bm}
\usepackage{graphicx}

\theoremstyle{plain}
\newtheorem{theorem}{Theorem}[section]
\newtheorem{cor}[theorem]{Corollary}
\newtheorem{lem}[theorem]{Lemma}
\newtheorem{pro}[theorem]{Proposition}
\newtheorem{Def}[theorem]{Definition}

\newcommand{\hym}{hyperbolic metric}
\newcommand{\ft}{forcing term}
\newcommand{\ub}{uniform bound}
\newcommand{\TS}{Teichm\"{u}ller space}

\newcommand{\kg}{Kleinian group}
\newcommand{\cmc}{constant mean curvature}

\newcommand{\mc}{mean curvature}
\newcommand{\scs}{small curvatures}
\newcommand{\pc}{principal curvature}

\newcommand{\mcf}{mean curvature flow}

\newcommand{\RS}{Riemann surface}
\newcommand{\sff}{second fundamental form}
\newcommand{\hf}{height function}
\newcommand{\naf}{nearly almost Fuchsian}
\newcommand{\af}{almost Fuchsian}

\newcommand{\qf}{quasi-Fuchsian}
\newcommand{\gradf}{gradient function}

\newcommand{\tm}{three-manifold}
\newcommand{\ms}{minimal surface}
\newcommand{\nf}{normal flow}
\newcommand{\ee}{evolution equation}
\newcommand{\ef}{equidistant foliation}
\newcommand{\is}{incompressible surface}
\newcommand{\es}{evolving surface}
\newcommand{\rs}{reference surface}
\newcommand{\ls}{limiting surface}
\newcommand{\ins}{initial surface}
\newcommand{\vp}{volume preserving}

\newcommand{\gbar}{\bar{g}}

\newcommand{\nablabar}{{\overline{\nabla}}}

\newcommand{\n}{\bm{n}}

\newcommand{\C}{\mathbb{C}}
\newcommand{\R}{\mathbb{R}}
\newcommand{\I}{\mathbb{I}}
\renewcommand{\H}{\mathbb{H}}

\newcommand{\Hscr}{\mathscr{H}}

\newcommand{\biginner}[2]{\left\langle #1\,,#2\right\rangle}
\newcommand{\inner}[2]{\langle #1\,,#2\rangle}
\newcommand{\ddl}[2]{\frac{d{#1}}{d{#2}}}

\newcommand{\ppl}[2]{\frac{\partial{#1}}{\partial{#2}}}

\renewcommand{\div}{\mathop{\mathrm{div}}}

\numberwithin{equation}{section}

\allowdisplaybreaks

\def\XXint#1#2#3{{\setbox0=\hbox{$#1{#2#3}{\int}$}
    \vcenter{\hbox{$#2#3$}}\kern-.5\wd0}}

\makeatletter
\def\@citestyle{\m@th\upshape\mdseries}
\def\citeform#1{{\bfseries#1}}
\def\@cite#1#2{{%
  \@citestyle[\citeform{#1}\if@tempswa, #2\fi]}}
\@ifundefined{cite }{%
  \expandafter\let\csname cite \endcsname\cite
  \edef\cite{\@nx\protect\@xp\@nx\csname cite \endcsname}%
}{}
\makeatother

\begin{document}

\title[Surfaces of Prescribed Mean Curvature]
{Surfaces of Prescribed Mean Curvature in
 Quasi-Fuchsian Manifolds}

\author{Zheng Huang}
\address{Department of Mathematics,
The City University of New York,
Staten Island, NY 10314, USA.}
\email{zheng.huang@csi.cuny.edu}

\author{Biao Wang}
\address{Department of Mathematics,
University of Toledo, Toledo, OH 43606, USA.}
\email{biao.wang@utoledo.edu}

\date{January 6, 2010}

\subjclass[2000]{Primary 53C44, Secondary 53A10, 57M05}

\begin{abstract}
Let $M$ be a {\qf} {\tm} that contains a closed {\is} with {\pc}s
within the range of the unit interval, for a prescribed function $\Hscr$ 
(with mild conditions) on $M$, we construct a closed {\is} with
{\mc} $\Hscr$. A direct application is the existence of embedded
surfaces of prescribed {\cmc}s with constants in $(-2,2)$. 
\end{abstract}

\maketitle
\section {Introduction}
\subsection{Main results}
By hyperbolic space, we mean a Riemannian manifold of constant
sectional curvature $-1$. We are particularly interested in the dimension three. The {\qf} manifold 
is an important class of hyperbolic {\tm}s, and
the space of these manifolds is the {\it {\qf} space}, a complex
manifold of (complex) dimension $6g-6$,
where we always assume any {\is} in a {\qf} manifold has genus
$g \ge 2$. Graphically, the {\qf} space
can be viewed naturally as a ``higher" {\TS}, a square with the
Fuchsian locus sitting inside as a diagonal.

\begin{Def}
A closed surface $S$ is called to have {\scs} if its {\pc}s
$\{\lambda_j\}$ satisfy
\begin{equation} \label{small}
   |\lambda_j(x)| < 1, \forall x \in S, j=1,2.
\end{equation}
\end{Def}
It is of great interest to understand the structures, such as
Riemannian, topological, and complex structures,
of the {\qf} space. A special subspace, the {\it {\af}} space,
consists of what we call {\af} manifolds:

\begin{Def}A {\qf} manifold is called {\bf \af} if it contains an
incompressible {\ms} of {\scs} in the sense of \eqref{small}.
Here an {\is} is a smooth closed surface in a {\tm} which
induces an injection between their fundamental groups.
\end{Def}

It is evident that {\is}s of {\scs} play important role in three
dimensional geometry and topology, see for
example \cite{Thu82}, \cite{Eps84}, \cite{Rub05} and \cite{GHW09}.
The space of {\af} manifolds is an open
subspace of the same dimension of the {\qf} space, and it contains
the space of Fuchsian manifolds which is diffeomorphic to {\TS}
(\cite{Uhl83}). Recent progress in GHMC (globally hyperbolic 
maximally compact) $AdS_3$ space (\cite{Mes07, KS07}) resembles 
strongly to the geometry of {\qf} {\tm}s.

A fundamental fact about a {\qf} {\tm} $M$ is that it admits an
immersed incompressible {\ms} (\cite{SY79},
\cite{SU82}). When $M$ is further {\af}, Uhlenbeck showed that
$M$ only admits one incompressible {\ms},
and it is necessarily embedded (\cite{Uhl83}). Furthermore, we are
able to use the {\mcf} to deform a
somewhat arbitrary graphical surface of a surface of {\scs} to
this {\ms} with exponential convergence (\cite{HW09b}).

We consider a larger class of {\qf} {\tm}s and obtain results of
existence of embedded {\is}s of special geometry:

\begin{Def}
A {\qf} manifold $M$ is called {\bf {\naf}} if it contains an
{\is} of {\scs} in the sense of \eqref{small}.
\end{Def}

Obviously this class contains all {\af} manifolds, though the number
of incompressible {\ms}s might be larger than one.

\vskip 0.1in
\noindent
{\bf Basic assumptions:} Throughout this paper, we denote $M$
a {\naf} manifold, and $S$ a closed {\is}
of $M$ which has {\scs} in the sense of \eqref{small}.
All surfaces we encounter are assumed to be
closed, incompressible, of genus at least two.

\vskip 0.1in
We are greatly inspired by a beautiful paper of Ecker-Huisken
(\cite{EH91}), where they studied the
problem of prescribed {\mc} in cosmological space-times. They
settled the problem with assumptions of the existence of the
barrier surfaces and certain monotonicity conditions on the
prescribed function, as well as a time-like convergence condition
on the space-times. As they pointed out, the {\mcf} in a
Lorentzian manifold behaves more regularly than in a
Riemannian manifold.

In the present work, we construct an {\is} in $M$ with
prescribed {\mc} function $\Hscr: M \rightarrow \R$
which satisfies some mild conditions. To make it more precise,
let $S$ be any {\is} in $M$, with
$|\lambda_{j}(S)| <1$, where $\{\lambda_j(S)\}_{j=1,2}$ are {\pc}s
of $S$. The {\nf} from $S$ forms a
foliation of parallel surfaces (or the {\ef}) $\{S(r)\}_{r \in \R}$,
see (\cite{Uhl83} or \cite {HW09a}), for the
{\tm} $M$. Therefore any point $P \in M$ can be uniquely represented
as $P(x,r)$, where $x$ is the conformal coordinate on $S$.

\begin{theorem}
Let $M$ be a {\naf} manifold and it contains a closed surface
$S$ with $|\lambda_{j}(S)| <1$ for $j =1,2$.
Let $\Hscr: M \to \R, (x,r) \mapsto \Hscr(x,r)
\subset [a,b] \subset (-2,2)$ be a smooth function with
bounded gradients $|\nablabar \Hscr|$ and
$|\nablabar^2 \Hscr|$ in $M$, then there exists an embedded
{\is} $S(\infty)$ on $M$ such that
$H(S(\infty)) = \Hscr |_{S(\infty)}$.
\end{theorem}

An immediate application is the following corollary:

\begin{cor}
For any constant $c \in (-2,2)$, there exists an embedded {\is}
$S(\infty)$ on $M$ with {\cmc} $c$. In
particular, $M$ admits an embedded incompressible {\ms}.
\end{cor}

By using the usual {\mcf}, we (\cite{HW09b}) have shown that
if $M$ is {\naf}, then $M$ admits (at least) one embedded {\ms}.

This paper is a continuation of our study of geometric {\ee}s in
{\qf} {\tm}s, begun in \cite{Wan08} and
\cite{HW09a}. The geometric evolution equations of the evolution
of hypersurfaces by their {\mc} have
been studied extensively in various ambient Riemannian manifolds,
(see for instance \cite{Bra78},
\cite{Hui84}, \cite{Hui86}, \cite{And02}, and many others) as well
as in Lorentzian spaces (see
\cite {Ger83}, \cite {Bar84}, \cite{EH91}, \cite{Eck03}, etc.).

The {\mcf} equation of a {\ft} $f$ has the following form:
\begin{equation}\label{mcf}
   \left\{
   \begin{aligned}
      \ppl{}{t}\,F(x,t)&=(f(x,t)-H(x,t))\nu(x,t)\ ,\\
      F(\cdot,0)&=F_{0}\ ,
   \end{aligned}
   \right.
\end{equation}
where all terms will be made transparent in section \S 2.3.
Obviously, this formulation is a generalization
of the {\mcf} (when the {\ft} $f \equiv 0$), and the {\vp} {\mcf}
(when $f = {\frac{\int_{S(t)}Hd\mu}{|S(t)|}}$,
the average {\mc} of the {\es} $S(t)$). The proof of the
Theorem 1.4 is based on solving the initial value
problem \eqref{mcf} with the prescribed function $\Hscr(F(\cdot,t))$
as the {\ft}. It should be noted that
the presence of a global {\ft} makes quantities of the {\ee}s
involved in the calculations more delicate.

Our proof of the main Theorem 1.4 roughly goes as follows: we start
with a fiber $S(r)$ of the {\ef}
$\{S(r)\}_{r \in \R}$ of $M$, which is a graph over a fixed surface
$S$ of small curvatures, then use it as
an initial immersion of the {\mcf} \eqref{mcf} with
$\Hscr(F(\cdot,t))$ as the forcing term, and we show
the long time existence of the solution for the equation.
The limiting surface $S(\infty)$ is embedded since
it is also a graph over $S$. Our estimates rely on the basic fact
that the graph functions behave quite
regularly in hyperbolic spaces under evolution equations (see for
example \cite {EH89, Unt03, CM04}).

\subsection{Notation}
This subsection recalls notation that will be employed in our paper.
The notation is quite similar to that
introduced in \cite{HW09b}, with a few additions:
\begin{itemize}
   \item $\H^3$: hyperbolic 3-space, with isometry group $PSL(2,\C)$;
   \item $M$: a {\naf} {\tm} (Definition 1.3);
   \item $S$: a surface of $M$ of {\scs} in the sense of \eqref{small};
   \item $\lambda_{j}(S)$: the {\pc}s of $S$, and
         $|\lambda_{j}(S)| < 1, j =1,2$;
   \item $\n$: the unit normal vector field to $S$;
   \item $S(r)$: the parallel surface to $S$ of hyperbolic
         distance $r$; these surfaces form the {\ef} of $M$;
   \item $\mu_{j}(x,r)$: the {\pc}s of $S(r)$ at the point
         $(x,r)$, $j=1,2$;
   \item $S(t)$: the {\es} under the {\mcf} \eqref{mcf};
   \item $g(\cdot,t)=\{g_{ij}\}$: the induced metric of $S(t)$;
   \item $A(\cdot,t)=\{h_{ij}\}$: the {\sff} of $S(t)$;
   \item $H(\cdot,t)=g^{ij}h_{ij}$: the {\mc} of $S(t)$ with
         respect to the normal pointing to $S$;
   \item $|A|^{2}=g^{ij}g^{kl}h_{ik}h_{jl}$: the square norm of
         $A(t)$ for $S(t)$;
   \item $\Delta=g^{ij}\nabla_{i}\nabla_{j}$: the Laplacian on $S(t)$;
   \item $\nabla$: the covariant derivative of $S(t)$;
  \item $\nu(t)$: the unit normal vector field to $S(t)$;
  \item $d\mu$: the area element for $S(t)$;
  \item $|S(t)|$: the surface area of $S(t)$;
  \item $u(\cdot,t)$: the {\hf} measuring the distance to the
        reference surface;
  \item $\Theta(\cdot,t)=\langle{\nu(\cdot,t)},{\n}\rangle$:
        the {\gradf}.
\end{itemize}
We add a bar on top for each quantity or operator with respect to
$(M,\gbar_{\alpha\beta})$.

\subsection*{Plan of the paper}
We provide necessary background material
in \S 2, especially the {\af}
manifolds, the {\ef}, and the {\mcf}. We prove the Theorem 1.4
(prescribing {\mc} function) in \S 3
(long-time existence) and \S 4 (convergence), by the way of
proving Theorem 3.1.

\subsection*{Acknowledgements}
The authors wish to thank Ren Guo for
many stimulating discussions.
They also thank Zhou Zhang for many suggestions, and especially
for his criticism. The research of
the first named author is partially supported by a PSC-CUNY grant.

\section{Preliminaries}
In this section, we briefly introduce some preliminary facts that
will be used in this paper. We start with
the geometry of {\qf} {\tm} in \S 2.1, and define the
{\ef}s for $M$ in \S 2.2, which is our point of
departure for our analysis; introduce
the general {\mcf} equations and their
corresponding evolution equations for the metrics and the
{\sff}s in \S 2.3.

\subsection{Quasi-Fuchsian {\tm}s}
For any unexplained terminology and more detailed references on
{\kg}s and low dimensional topology, we refer
to \cite{Mar74, Thu82}.

The universal cover of a hyperbolic {\tm} is $\H^3$, and the
deck transformations induce a representation of the
fundamental group of the manifold in $Isom(\H^3)= PSL(2,\C)$,
the (orientation preserving) isometry group of
$\H^3$. A subgroup $\Gamma \subset PSL(2,\C)$ is called a
{\em Kleinian group} if $\Gamma$ acts on $\H^{3}$
properly discontinuously.

For any {\kg} $\Gamma$, $\forall\,p\in\H^{3}$, the orbit set
$\Gamma(p)=\{\gamma(p)\ |\ \gamma\in \Gamma\}$
has accumulation points on the boundary
$S^{2}_\infty=\partial\H^{3}$, and these points are the
{\em limit points}  of $\Gamma$, and the closed set of all these
points is called the {\em limit set} of $\Gamma$, denoted by
$\Lambda_{\Gamma}$. In the case when $\Lambda_{\Gamma}$ is contained
in a circle $S^1 \subset S^2$, the quotient $\H^{3}/\Gamma$ is called
{\em Fuchsian}, and it is isometric to a product space of a
totally geodesic hyperbolic surface and the real line. Clearly,
the space of Fuchsian manifolds is isometric to the space of {\hym}s
on a closed surface, or {\TS}.

If the limit set $\Lambda_{\Gamma}$ lies in a Jordan curve, the
quotient $\H^{3}/\Gamma$ is called {\em {\qf}}, and
it is topologically a product space $S \times \R$, where $S$ is
a closed surface. A {\qf} is a complete hyperbolic
{\tm} which is quasi-isometric to a Fuchsian manifold.

An extraordinary fact about {\qf} manifolds is the Bers'
simultaneous unformization theorem (\cite{Ber72}), which
states that, for each pair of points in {\TS}, there is one
{\qf} manifold with these two conformal structures as
conformal boundaries. Therefore, the deformation theory of
{\kg}s is deeply related to the deformation theory of {\RS}s.

We are particularly interested in a class of {\qf} manifolds:
{\naf} manifolds which contain a closed surface of
{\scs}. By the work of \cite{SY79} and \cite{SU82}, such a {\naf}
manifold $M$ admits at least one immersed
{\ms}. The importance of this particular class is two-fold:
constant curvature of the space simplifies the {\ee}s
involved in the calculation, and admission of the {\ef} provides
geometric bounds for the {\es}s in the {\mcf}.

\subsection{Equidistant foliation for $M$}
We now follow the notation introduced in \S 1.2. The induced metric
on the reference surface $S$ is given by
$g_{ij}(x)=e^{2v(x)}\delta_{ij}$, where $v(x)$ is a smooth function
on $S$, and the {\sff} of $S$ is
$A(x)=[h_{ij}]_{2\times{}2}$, here $h_{ij}$ is given by, for
$1\leq{}i,j\leq{}2$,
\begin{equation*}
   h_{ij}=\langle{\nablabar_{e_{i}}\nu'}, {e_{j}}\rangle
         =-\langle{\nablabar_{\nu'}e_{i}},{e_{j}}\rangle\ ,
\end{equation*}
where $\{e_{1},e_{2}\}$ is a basis on $S$, $\nu'$ is the unit
normal vector to $S$, and $\nablabar$ is the
Levi-Civita connection on $(M,\gbar_{\alpha\beta})$.

Let $\lambda_1(x)$ and $\lambda_2(x)$ be the eigenvalues of $A(x)$
and $|\lambda_j(x)| < 1$ for $j=1,2$.
They are the {\pc}s of $S$, and we denote
$H(x) =\lambda_1(x)+ \lambda_2(x)$ as the {\mc} function of $S$,
and $|H(S)|<2$.

Let $S(r)$ be the family of equidistant surfaces with respect
to $S$, i.e., $S(r)$ is a parallel surface to $S$ by
moving hyperbolic distance $r$ away from $S$ in the normal direction:
\begin{equation*}
   S(r)=\{\exp_{x}(r\nu)\ |\ x\in{}S\}\ ,
   \quad{}r\in(-\varepsilon,\varepsilon)\ .
\end{equation*}
The induced metric on $S(r)$ is denoted by $g(x,r) = g_{ij}(x,r)$,
and the {\sff} is denoted by
$A(x,r)=[h_{ij}(x,r)]_{1\leq{}i,j\leq{}2}$. The {\mc} on $S(r)$
is thus given by $H(x,r)=g^{ij}(x,r)h_{ij}(x,r)$. We
collect the following lemma:

\begin{lem}[\cite {Uhl83}, \cite{HW09a}]
The induced metric $g(x,r)$ on $S(r)$ has the form
\begin{equation}\label{eq:metric of S(r)}
   g(x,r)=e^{2v(x)}[\cosh{r}\I+\sinh{r}
   e^{-2v(x)}A(x)]^{2}\ ,
\end{equation}
The metric is of non-singular for all $r \in \R$. The {\pc}s of the
surface $S(r)$ are given by
\begin{equation}\label{pc4ef}
   \mu_{j}(x,r)=\frac{\tanh{}r+\lambda_{j}(x)}
   {1+\lambda_{j}(x)\tanh{}r}\ ,\qquad j=1,2\ .
\end{equation}
and mean curvature is
\begin{equation} \label{mc4ef}
   H(x,r)=\frac{2(1+\lambda_{1}\lambda_{2})\tanh{}r+
   (\lambda_{1}+\lambda_{2})(1+\tanh^{2}r)}
   {1+(\lambda_{1}+\lambda_{2})\tanh{}r+
   \lambda_{1}\lambda_{2}\tanh^{2}r}\ .
\end{equation}
\end{lem}

Therefore $\{S(r)\}_{r\in\R}$ forms a foliation of surfaces parallel
to $S$, called the {\ef} or the {\nf}. It is easy to verify

\begin{lem}
For the {\ef} $\{S(r)\}_{r\in\R}$, we have the following:
\begin{enumerate}
   \item Each $S(r)$ has {\scs}: $|\mu_j(x,r)| < 1$ for all
         $x \in S$, $r \in \R$;
   \item For fixed $x \in S$, $\mu_j(x,r)$ is an increasing
         function of $r$. Moreover, $\mu_j(x,r) \to \pm 1$ as
         $r \to \pm\infty$;
   \item For fixed $x \in S$, $H(x,r)$ is an increasing function
         of $r$. Moreover, $H(x,r) \to \pm 2$ as $r \to \pm\infty$.
\end{enumerate}
\end{lem}

We note that, if $M$ is not {\naf}, then the metric $g(x,r)$ on
$S(r)$ will develop singularity quickly.

The following lemma is the well-known Hopf's maximum principle
for tangential hypersurfaces in Riemannian geometry:

\begin{lem}[\cite{Hop89}] \label{max}
Let $\Sigma_{1}$ and $\Sigma_{2}$ be two hypersurfaces in a
Riemannian manifold which intersect at a common point $p$ tangentially.
If $\Sigma_{2}$ lies in positive side of $\Sigma_{1}$ around $p$, then 
$H_{1} \le H_{2}$, where $H_{i}$ is the {\mc} of
$\Sigma_{i}$ at $p$ for $i=1,2$.
\end{lem}

\subsection{Mean curvature flow with a forcing term}
Let $F_{0}: S\to{}M$ be the immersion of $S$ in $M$ such that
$S_{0}=F_{0}(S)$ is contained in the
positive side of $S$, and is a graph over $S$ with respect to $\n$,
i.e., $\langle{\n},{\nu}\rangle \geq c>0$, here $\n$ is the unit
normal vector on $S$ and $\nu$ is the unit normal vector on $S_{0}$
and $c$ is a constant depending only on $S_{0}$.

We consider a family of immersions of surfaces in $M$,
\begin{equation*}
   F:S\times[0,T)\to{}M\ , \quad{}0\leq{}T\leq\infty
\end{equation*}
with $F(\cdot,0)=F_{0}$. For each $t\in[0,T)$,
$S(t)=\{F(x,t)\in{}M\ |\ x\in{}S\}$ is the evolving surface
at time $t$, and $H(x,t)$ its {\mc}.

The {\mcf} equation (\cite{EH91}) with a forcing term $f$ is
given by, as in \eqref{mcf}:
\begin{equation*}
   \left\{
   \begin{aligned}
      \ppl{}{t}\,F(x,t)&=(f(x,t)-H(x,t))\nu(x,t)\ ,\\
      F(\cdot,0)&=F_{0}\ ,
   \end{aligned}
   \right.
\end{equation*}
Here $-\nu$ points to the surface $S$.

The equation \eqref{mcf} is parabolic, and Huisken proved the
short-time existence of the solutions
and found initial compact surface quickly develops singularities
along the flow, moreover, he showed
the blow-up of the norm of the {\sff}s if the singularity occurs
in finite time.

\begin{theorem}[\cite{Hui84}, \cite{Hui86}]
If the initial surface $S_0$ is smooth, then the equation \eqref{mcf}
has a smooth solution on some maximal open time interval $0\leq{}t<T$,
where $0<T \leq \infty$. If $T<\infty$, then
$|A|_{\max}(t)\equiv\max\limits_{x\in{}S}|A|(x,t)\to\infty$\,
as $t\to{}T$.
\end{theorem}

Therefore, the key to the existence of long time solution is to
find {\ub}s for the square norm of the
{\sff}s on the {\es}s along the flow.

\section{Prescribing {\mc}}
This section is devoted to proving the existence of long-time
solution part of the Theorem 1.1, by analyzing
the {\mcf} equation with the prescribed {\mc} function as the
forcing term:
\begin{equation}\label{mcf2}
   \left\{
   \begin{aligned}
      \ppl{}{t}\,F(x,t)&=(\Hscr-H(x,t))\nu(x,t)\ ,\\
      F(\cdot,0)&=F_{0}\ ,
   \end{aligned}
   \right.
\end{equation}
where $\Hscr$ is the prescribed {\mc} function on
$M = \{S(r)\}_{r\in\R}$. Writing any point $P \in M$ as
a pair $(x,r)$, where $x \in S$, the function $\Hscr$ is smooth,
and with bounded gradient with respect to
$\nablabar$. Moreover, we require $-2 < a \le \Hscr \le b < 2$
for some constants $a$ and $b$. Our strategy
is to show the long-time existence of the solution to
\eqref{mcf2}, as well as the convergence and the
uniqueness of the limiting surface.

Our major goal is to establish the following theorem, which will
imply our main theorem 1.4:

\begin{theorem}
Let $M$ and $S$ be as in Theorem 1.4. Suppose a smooth closed
surface $S_0 = S(r)$ for some $r$, where
$\{S(r)\}_{r \in \R}$ forms the {\ef} for $M$ as in \S 2.2. Then
\begin{enumerate}
   \item the {\mcf} equation \eqref{mcf2} with initial surface
         $S(0) = S_0$ has a long time solution;
   \item the {\es}s $\{S(t)\}_{t \in \R}$ stay smooth and they remain
         as graphs over $S$ for all time;
   \item For every sequence $\{t_i\} \to \infty$, there is a
         subsequence $\{t_i'\} \to \infty$ such that the surfaces
         $\{S(t_i)\}$ converges uniformly in $C^{\infty}$ to a
         smooth embedded surface $S(\infty)$ satisfying
         $H(S(\infty)) = \Hscr |_{S(\infty)}$.
\end{enumerate}
\end{theorem}

\subsection{Some evolution equations}
In this subsection, we collect and derive a number of evolution
equations of some quantities and operators
on $S(t)$, $t\in[0,T)$ which are involved in our calculations.

Proceeding as in \cite[Lemmas 3.2, 3.3, and Theorem 3.4]{Hui84}, and
keeping track of terms involving $\Hscr$, we find:

\begin{pro}
The evolution equations of the induced metric $g_{ij}$, the normal
vector field $\nu$, and the area element
$d\mu$ are given by
\begin{align}
   \ppl{}{t}\,g_{ij}&=2(\Hscr-H)h_{ij}\ ,\\
   \ppl{}{t}\,\nu&=\nabla(H - \Hscr)\ ,\\
   \ppl{}{t}\,d\mu&=H(\Hscr-H)d\mu\,
\end{align}
\end{pro}

We will need the {\ee} for the {\mc} $H(\cdot,t)$:
\begin{lem}
\begin{equation}\label{ee-h}
   \ppl{}{t}H=\Delta(H-\Hscr)+(H-\Hscr)(|A|^{2}-2).
\end{equation}
\end{lem}
\begin{proof}
Consider the well-known Simons' identity (see \cite{Sim68,SSY75}),
satisfied by the {\sff} $h_{ij}$:
\begin{equation*}
   \Delta{}h_{ij}=\nabla_{i}\nabla_{j}H-
   (|A|^{2}-2)h_{ij}+ H(h_{il}h_{lj}+g_{ij})\ .
\end{equation*}
Here we used our special situation: $M$ has constant sectional
curvature $-1$, and $M$ has dimension
three, hence Ricci curvature $Ric(\nu,\nu) = -2$, and
$\bar{R}_{3i3j} = -g_{ij}$ for $1\le i,j\le 2$.

Also we have
\begin{equation*}
   \ppl{}{t}h_{ij}=\nabla_{i}\nabla_{j}(H -\Hscr) +
   (\Hscr - H)(h_{il}h_{lj}+g_{ij}).
\end{equation*}
Now we proceed the calculation using $H = g^{ij}h_{ij}$ and
$(3.2)$ to obtain \eqref{ee-h}.
\end{proof}

We are now arriving at the key {\ee}:
\begin{lem}We have the following equation for the square norm of
the {\sff}:
\begin{align}\label{ee-a}
   \ppl{}{t}\,|A|^{2}=&\, \Delta|A|^{2}-2|\nabla{}A|^{2}-
         2h_{ij}\nabla_{i}\nabla_{j}\Hscr\\
    \notag
      &\,+2|A|^{2}(|A|^2+2)-2\Hscr Tr{}A^{3}+2H(\Hscr-2H)\ ,
\end{align}
where $Tr(A^{3})=h_{ij}h_{jl}h_{li}$.
\end{lem}

\begin{proof}
Using again the fact that $M$ is a hyperbolic {\tm}, the
equation \eqref{ee-a} is obtained by proceeding the
calculations in \cite{Hui86} and \cite{HY96}.
\end{proof}

It will be very important for our a priori estimates for
$|A|^2$ to recall the {\it height function} $u(\cdot,t)$ and
the {\it gradient function} $\Theta(\cdot,t)$ on $S(t)$
from the introduction:
\begin{align}
   u(x,t)&=\ell(F(x,t))\\
   \Theta(\cdot,t)&=\langle{\nu(\cdot,t)},{\n}\rangle\ ,
\end{align}
for all $(x,t)\in{}S\times[0,T_{\max})$. Here $T_{\max}$ is the
right endpoint of the maximal time interval on which the
solution to \eqref{mcf2} exists, and $\ell(p) = \pm dist(p,S)$
for all $p \in M$, the distance to the reference surface $S$.
It is clear that the surface $S(t)$ becomes a graph over $S$ if
$\Theta(\cdot,t) > 0$ on $S(t)$.

The {\ee}s of $u(\cdot,t)$ and $\Theta(\cdot,t)$ have the
following forms:
\begin{pro}[\cite {Bar84},\cite{EH91}]
\begin{align}\label{ee-hf}
   \ppl{}{t}\,u
        =&\,(\Hscr-H)\Theta\\
      \label{ee-hf2}
        =&\,\Delta{}u-\div(\nablabar\ell)+\Hscr\Theta\ ,\\
        \label{ee-gf}
   \ppl{}{t}\,\Theta=
       &\,\Delta{}\Theta+(|A|^{2}-2)\Theta+\n(H_{\n})
          -\inner{\n}{\nabla\Hscr}\\
       \notag
       &\,+(\Hscr-H)\inner{\nablabar_{\nu}\n}{\nu}\ ,
\end{align}
where $\div$ is the divergence on $S(t)$, and $\n(H_{\n})$ is
the variation of {\mc} function of $S(t)$ under
the deformation vector field $\n$.
\end{pro}

\begin{proof} We only show \eqref{ee-gf}. By the definition of the
{\gradf}, we have
\begin{equation*}
   \ppl{\Theta}{t}=\inner{\n}{\nabla(H-\Hscr)}+
   (\Hscr-H)\inner{\nablabar_{\nu}\n}{\nu}\ .
\end{equation*}
On the other hand, we have
\begin{equation*}
   \Delta{}\Theta=-(|A|^{2}-2)\Theta+
   \inner{\n}{\nabla{}H}-\n(H_{\n})\ ,
\end{equation*}
ompleting the proof.
\end{proof}

\subsection{Estimates for $\Theta(\cdot,t)$}
It appears to be very difficult to bound $|A|^2$ simply from the
equation \eqref{ee-a}. In what follows, we
show the positivity of the {\gradf} $\Theta(\cdot,t)$ and use the
{\ee} for $\Theta^{-\delta}(\cdot,t)$, for
some $\delta>2$ to add enough negative terms to \eqref{ee-a}.
This subsection is to serve this purpose.

Our first technical lemma is the following: given positive lower
bound on the gradient function on the {\ins},
the function stays positive along the {\mcf}, i.e.,

\begin{pro} \label{bd4gf}
If $\Theta(\cdot,0)\geq{}C>0$, where $C$ is any positive constant,
then $\Theta(\cdot,t)\ge \Theta_0 >0$
for some constant $\Theta_0$, depending only on $S_0$ and $T$,
for $t \in[0,T)$.
\end{pro}

The statement of this lemma is slightly stronger than what we
need: our choice of the {\ins} $S(0)$ is a
parallel surface of the {\rs} $S$, hence $\Theta(\cdot,0) = 1$.

\begin{proof} We proceed as in [\cite{HW09b}, Lemma 3.4], with
special cares on terms involving $\Hscr$.
Let $\Theta_{\min}(t)=\min_{x\in{}S}\Theta(x,t)$.

We estimate the terms in the equation \eqref{ee-gf}, starting
with the expression $\n(H_{\n})$, from
(\cite[Eq. (2.10)]{Bar84}):
\begin{equation}
   |\n(H_{\n})|\leq{}C_{1}(\Theta^{3}+\Theta^{2}|A|)\ ,
\end{equation}
for some $C_1 > 0$.

We also have the following estimate from (\cite[Page 187]{Eck03})
\begin{equation}
   |\inner{\nablabar_{\nu}\n}{\nu}|\leq{}C_{2}\Theta^{2}\ ,
\end{equation}
where $C_{2}=\|\nablabar\n\| > 0$.

Thirdly, we have the following estimate from (\cite[Page 602]{EH91}):
\begin{equation}
   |\inner{\n}{\nabla\Hscr}| \leq
   \Theta^{2}\|\nablabar\Hscr\| \le \|\nablabar\Hscr\|,
\end{equation}
which is bounded since $\Hscr$ has bounded gradient in $M$.

Collecting these estimates, and we obtain from the equation \eqref{ee-gf}:
\begin{align*}
   \ddl{}{t}\,\Theta_{\min}
      \geq&\,(|A|^{2}-2)\Theta_{\min}-
             C_{1}(\Theta_{\min}^{3}+|A|\Theta_{\min}^{2})\\
          &\,-\|\nablabar\Hscr\|\Theta_{\min}^{2}-
              C_{2}(|\Hscr|+|H|)\Theta_{\min}^{2}\\
      \geq&\,\big((|A|^{2}-2)-C_{1}(1+|A|)-
             \|\nablabar\Hscr\|-
             C_{2}(|\Hscr|+\sqrt{2}\,|A|)\big)\Theta_{\min}\\
         =&\,\big(|A|^{2}-(C_{1}+C_{2}\sqrt{2})|A|-
             (2+C_{1}+\|\nablabar\Hscr\|+
            |\Hscr|C_{2})\big)\Theta_{\min}\\
      \geq&\,-\left(\frac{(C_{1}+C_{2}\sqrt{2})^{2}}{4}+
             2+C_{1}+\|\nablabar\Hscr\|+|\Hscr|C_{2}
              \right)\Theta_{\min}
\end{align*}
Since $\Theta_{\min}(0)\geq{}C>0$, then
\begin{equation*}
   \Theta_{\min}(t)\geq{}C\exp(-C_{3}t)
   \quad\text{on}\ [0,T)\ ,
\end{equation*}
where
\begin{equation*}
   C_{3}=\frac{(C_{1}+C_{2}\sqrt{2})^{2}}{4}+
         2+C_{1}+\|\nablabar\Hscr\|+|\Hscr|C_{2}\ .
\end{equation*}
We can choose $\Theta_0 = C\exp(-C_{3}T)$ to complete the proof.
\end{proof}
We also record the following proposition on the derivatives of
$\Theta(\cdot,t)$, postponing its proof in the
next subsection as we need extra bounds on $u(\cdot, t)$ and
its derivatives.

\begin{pro}\label{bd-gf2}
Suppose the {\mcf} equation \eqref{mcf2} has a solution on $[0,T)$,
$0<T\leq\infty$, then there exists a
constant $0<C_4<\infty$ depending only on $S_{0}$ such that
\begin{equation*}
   |\nabla\Theta|^{2}\leq{}C_4
\end{equation*}
on $S(t)$, for $0\leq{}t<T$.
\end{pro}

\subsection{Estimates on $u(\cdot,t)$}
We now turn our attention to the following estimates on the height
function, where we rely on, and make
use of the properties of the {\mc}s along the {\ef} on $M$, as well
as the bounds $[a,b] \subset (-2,2)$ on the prescribed {\mc} $\Hscr$.

\begin{theorem}\label{bd4hf}
Suppose the {\mcf} \eqref{mcf2} has a solution on $[0,T)$,
$0<T\leq\infty$, then $u(\cdot,t)$ is
uniformly bounded on $S\times[0,T)$, i.e.,
\begin{equation*}
   0<C_{5}\leq{}u(x,t)\leq{}C_{6}<\infty\ ,
   \quad\forall\,(x,t)\in{}S\times[0,T)\ ,
\end{equation*}
where $C_{5}$ and $C_{6}$ are constants depending only on the
surface $S_{0}$.
\end{theorem}

\begin{proof} At each time $t\in[0,T)$, let $x(t)\in{}S$ be the
point such that
\begin{equation*}
   u_{\max}(t)\equiv\max_{x\in{}S}u(x,t)=u(x(t),t)\ ,
\end{equation*}
and let $y(t)\in{}S$ be the point such that
\begin{equation*}
   u_{\min}(t)\equiv\min_{y\in{}S}u(y,t)=u(y(t),t)\ .
\end{equation*}
By the equation \eqref{ee-hf} and the positivity of $\Theta$
along the flow (Lamma 3.6), the part of $S(t)$
with $H<\Hscr|_{S(t)}$ will move along the positive direction of
$\n$ while the part of $S(t)$ with
$H>\Hscr|_{S(t)}$ will move along the negative direction of $\n$,
therefore we can assume that
$u_{\max}(t)$ is increasing and $u_{\min}(t)$ is decreasing, after
some $t_0 \in (0,T)$.

In oder to show that $u$ is uniformly bounded along the flow
for $t\in[t_0,T)$, we must exclude the
following cases as $t\to{}T$:
\begin{enumerate}
   \item $u_{\min}(t)\to{}-\infty$ and $u_{\max}(t)\to\infty$;
   \item $u_{\min}(t)\to\infty$ and $u_{\max}(t)\to\infty$;
   \item $u_{\min}(t)\to{}-\infty$ and $u_{\max}(t)\to{}-\infty$;
   \item $u_{\min}(t)$ is uniformly bounded, while
         $u_{\max}(t)\to{}\infty$;
   \item $u_{\min}(t)\to{}-\infty$, while $u_{\max}(t)$ is
         uniformly bounded.
\end{enumerate}

We consider, at $F(x(t),t)$, $\Theta=\inner{\n}{\nu}=1$, then
\begin{equation*}
   0\leq\ppl{u}{t}=\Hscr-H\ .
\end{equation*}
By the maximum principle, we have
\begin{equation*}
   b\geq\Hscr(x(t),t)\geq{}H(x(t),t).
\end{equation*}
Now we recall from Lemma 2.2, the {\pc}s $\mu_j(p,r)$ of the fiber
surfaces for the {\ef} are increasing
functions of $r$ with limits $\pm 1$ as $r \to \pm\infty$ and the
{\mc}s approach $\pm 2$ as
$r \to \pm\infty$. Now if $u_{\max}(t)\to\infty$ as $t\to{}T$,
then we will have $b\geq{}2$, which
contradicts our assumption that $b\in (-2,2)$.

Similarly, at the point $F(y(t),t)$, we have
\begin{equation*}
   a\leq\Hscr(y(t),t)\leq{}H(y(t),t),
\end{equation*}
and if $u_{\min}(t)\to-\infty$ as $t\to{}T$, then we will have
$a\leq{}-2$. This is also impossible since
$a,b\in(-2,2)$.

So the {\mcf} is uniformly bounded by two parallel surfaces
$S(r_{1})$ and $S(r_{2})$ with
$0<r_{1}<r_{2}<+\infty$ on $[0,T)$.
\end{proof}
This theorem guarantees the {\es}s stay in compact region in $M$
along the entirety of the flow.
\begin{pro}
If the {\mcf} equation \eqref{mcf2} has a solution on $[0,T)$,
$0<T\leq\infty$, then
\begin{equation*}
   |\nabla^{\ell}u|\leq{}K_{\ell}< \infty ,
\end{equation*}
for all $\ell=1,2,\ldots$, where $\{K_{\ell}\}_{\ell=1}^{\infty}$
is the collection of constants depending only on
$\ell$, the initial data and the maximal time $T$.
\end{pro}
\begin{proof}
From Lemma 3.6, {\es}s are graphs of the {\hf} $u(\cdot,t)$ to
the {\rs} $S$. Therefore
$\Theta(\cdot, t)=1/\sqrt{1+|\nabla{}u|^{2}}$ (\cite{Hui86}),
and then $|\nabla{}u|$ is uniformly bounded from
above by a constant depending only on the initial data and $T$.

We observe that equation \eqref{ee-hf2} is a single quasilinear
parabolic equation for the {\hf} $u(\cdot,t)$, and
$u$ is uniformly bounded by the Theorem 3.8. This enables us to
apply the standard regularity results in
quasilinear second order parabolic equations
(\cite{Fri64,Lie96}) to find
\begin{equation*}
   |\nabla^{\ell}u|\leq{}K_{\ell}< \infty ,
\end{equation*}
for all $\ell=1,2,\ldots$, where $\{K_{\ell}\}_{\ell=1}^{\infty}$
is the collection of constants depending only on $\ell$, the initial
data and $T$.
\end{proof}

\begin{proof}[Proof of Proposition 3.7]
The conclusion is immediate by the relation
$\Theta(\cdot, t)=1/\sqrt{1+|\nabla{}u|^{2}}$,  and upper bounds
on the derivatives of $u(\cdot,t)$ (Proposition 3.9).
\end{proof}
\subsection{Estimates on $|A|^2$ and $H^2$}
In this subsection we establish {\ub}s for $|A|^2$, which is
crucial in showing the long-time existence of
the solution to the {\mcf} \eqref{mcf2}.

\begin{theorem}
If the mean curvature flow \eqref{mcf2} has a solution on $[0,T)$,
$0<T\leq\infty$, then there exists a
constant $C_{7}<\infty$ depending only on $S_{0}$ such that
\begin{equation*}
   |A(\cdot,t)|^{2}\leq{}C_{7}<\infty
\end{equation*}
on $S(t)$, for $0\leq{}t<T$. Moreover,
\begin{equation}
   \sup_{S_0 \times[0,T)}|\nabla^{m}A(\cdot,t)|^{2}
   \leq{}C(m)<\infty
\end{equation}
for all $m\geq{}1$.
\end{theorem}
\begin{proof}
In the first part, we follow closely to the proof of [\cite{HW09b},
Theorem 4.1], with special cares to terms involving $\Hscr$.

We introduce the function $\eta(\cdot,t) = {\frac{1}{\Theta(\cdot,t)}}$,
and we want to use the {\ee} for some power of $\eta(\cdot,t)$ to
add enough negative terms to the right-hand side of \eqref{ee-a}.
We have
\begin{eqnarray*}
   \ppl{}{t}\eta & = & -\eta^2\Theta_t \\
   &=& \Delta\eta - 2\Theta|\nabla\eta|^2 -
   (|A|^{2}-2)\eta -\eta^2 J',
\end{eqnarray*}
where we denote
$J' =\n(H_{\n})+ (\Hscr -H)\inner{\nablabar_{\nu}\n}{\nu} -
\inner{\nabla \Hscr}{\n} $.

We also have
\begin{equation}
   L(\eta^2) = -6|\nabla \eta|^2 - 2\eta^2(|A|^2-2) -2\eta^3 J',
\end{equation}
where we introduce the operator $L =  \ppl{}{t}-\Delta$ to
simplify our notation.

We also have
\begin{equation}
   L(\eta^4) = -20\eta^2|\nabla \eta|^2 -
   4\eta^4(|A|^2-2) -4\eta^5 J'.
\end{equation}

We consider the function $f(\cdot,t) = |A|^2 \eta^4$, and compute
\begin{equation}
   \ppl{}{t}f = \eta^4\ppl{}{t}(|A|^2) + |A|^2\ppl{}{t}(\eta^4),
\end{equation}

We apply \eqref{ee-a} and $(3.17)$ to above and find
\begin{eqnarray*}
   L(f)
     &=& -2\eta^4|A|^4 -2\eta^4|\nabla A|^2 +
            H\Hscr(3|A|^2 -H^2-2)\eta^4 -
            2h_{ij}\nabla_i\nabla_j\Hscr\\
     & &- 2\nabla|A|^2 \cdot \nabla(\eta^4) +
         4(3|A|^2-H^2)\eta^4-20|A|^2\eta^2|\nabla \eta|^2 -
         4|A|^2\eta^5 J',
\end{eqnarray*}
where the dominant term on the right is
$-2\eta^4|A|^4 = -2\Theta^4 f^2$, for large $|A|^2$.

Note that we applied the assumption that $\Hscr$ have bounded
gradients $|\nablabar \Hscr|$ and
$|\nablabar^2 \Hscr|$ in $M$, and the following bound
\cite[p. 604]{EH91}:
\begin{equation*}
   |\nabla^2 \Hscr| \le \Theta^2|\nablabar^2 \Hscr| +
   \Theta |A||\nablabar\Hscr|
   \le |\nablabar^2 \Hscr| + |A||\nablabar\Hscr|.
\end{equation*}

Now suppose $|A|^2$ is not uniformly bounded, then
\begin{equation*}
   |A|_{\max}^2(\cdot,t) \to \infty \quad
   \text{as}\ t \to T.
\end{equation*}
Since $f(\cdot,t) = |A|^2 \eta^4 \ge |A|^2$, we have:
\begin{equation*}
   f_{\max}(\cdot,t) \to \infty\quad
   \text{as}\ t \to T.
\end{equation*}

Therefore there exists a $t_1 \in (0,T)$ such that when $t > t_1$, we have
\begin{eqnarray*}
   \ddl{}{t}f_{\max} &\le& -2\eta^4|A|^4 -2\eta^4|\nabla A|^2 +
   4(3|A|^2 -H^2) \eta^4 \\
   &  & - 2\nabla|A|^2 \cdot \nabla(\eta^4) -
   12|A|^2\eta^2|\nabla \eta|^2 -4|A|^2\eta^5 J'.
\end{eqnarray*}
From Proposition 3.6, at the point where $f_{\max}$ occurs, we have
\begin{equation*}
    -2\eta^4|A|^4 = -2\Theta^4 f_{\max}^2 \le
    -2\Theta_0^4 f_{\max}^2.
\end{equation*}

From the proof of Proposition 3.6, we estimate the term $J'$:
\begin{eqnarray*}
|J'| &=& |\n(H_{\n})+ (\Hscr -H)
\inner{\nablabar_{\nu}\n}{\nu} -\inner{\nabla \Hscr}{\n}| \\
&\le& C_1\Theta^3+(C_1|A|+C_2(2+|H|)+\|\nablabar \Hscr\|)\Theta^2.
\end{eqnarray*}
Since $\Theta_0 \le \Theta < 1$, for $t>t_1$, now we have
\begin{equation*}
   \ddl{}{t}f_{\max} \le -2\Theta_0^4 f_{\max}^2 +
   \text{lower order terms}.
\end{equation*}
This is a contradiction since $\ddl{}{t}f_{\max} >0$.
Therefore $f(\cdot,t) = |A|^2 \eta^4$ is uniformly
bounded, which also bounds $|A|^2$ from above.

The second part of the theorem is a standard induction argument,
similar to the proof of \cite[Proposition 4.6]{HW09b}, where we
applied the argument in \cite{Ham82} and \cite{Hui84}.
\end{proof}

\section{Prescribing {\mc}: conclusion}
With all the pieces in place, we, in this section, prove Theorem 3.1:
we establish long-time existence,
examine convergence, and investigate the {\ls}s.
\begin{proof} of Theorem 3.1 (i) and (ii) (long-time existence):
Suppose the maximal time $T < \infty$, and
denote
\begin{equation}\label{ST}
   S(T) = \lim_{t \to T}S(t) = \{\lim_{t \to T}F(\cdot,t)\}.
\end{equation}
By the Theorem 3.6, the {\hf} $u(\cdot,t)$ is uniformly bounded,
therefore, the surfaces $\{S(t)\}_{t\in[0,T)}$
stay in a compact smooth region in $M$, hence \eqref{ST}
is well-defined.

Applying the {\ub} on $|A(\cdot,t)|^2$ (Theorem 3.10), and the
elementary inequality $|A|^2 \ge {\frac{1}{2}}H^2$, we find
that $|H(\cdot,t)|^2$ is also uniformly bounded. The bounds on
$|A(\cdot,t)|^2$, $|H(\cdot,t)|^2$ and $\Hscr$, together with
equation $(3.2)$, imply the following:
\begin{equation*}
   \int_0^{T}\max_{S(t)}|\ppl{}{t}g_{ij}(\cdot,t)|dt
   \le C_8 < \infty,
\end{equation*}
for some positive constant $C_8$. This enables us to apply
\cite[Lemma 14.2]{Ham82}, and find that
$S(T)$ is a surface. It is also smooth since all derivatives
of $A(\cdot,t)$ are bounded by the second part of
Theorem 3.10.

The {\es}s $\{S(t)\}$ stay as graphs of the {\rs} $S$ by
Propositions 3.6 and 3.7. Hence the limiting smooth
surface $S(T)$, is again, a graph over $S$. Now $S(T)$ satisfies
the initial conditions for the {\mcf} equation:
smooth, closed, incompressible, and graph over $S$, therefore we
use it as our {\ins} in the equation
\eqref{mcf2} to extend the flow beyond the maximal time $T$,
by the existence of short-time solutions for
the parabolic equation (Theorem 2.4).

Therefore, the solution to \eqref{mcf2} exists for all time,
i.e., $T= \infty$, and each {\es} $S(t)$ stays as a
graph over $S$.
\end{proof}

We are interested in the limiting behavior of the {\mc}
$H(\cdot,t)$:

\begin{theorem}
The following holds:
\begin{equation}\label{sup}
   \sup_{t \to \infty}|H(\cdot,t) - \Hscr| = 0.
\end{equation}
\end{theorem}

\begin{proof}
Let $M_{t}\subset{}M$ be the region bounded by the {\rs} $S$ and
$S(t)$. By Proposition 3.6, surfaces $\{S(t)\}$ are bounded in a
compact region, hence the area $|S(t)|$ and the volume $|M_t|$ are
uniformly bounded in $t$.

Applying the divergence theorem, we compute:
\begin{equation}
   \ddl{}{t}\int_{M_{t}}\Hscr dV =
   \int_{M_{t}}\overline{\div}\left(\Hscr\ppl{}{t}\,F\right) dV=
   \int_{S(t)}\biginner{\ppl{}{t}F}{\nu}\Hscr d\mu,
\end{equation}
here $\overline{\div}$ is the divergence on $M$. Denote the function
\begin{equation*}
   \alpha(t) = |S(t)|-\int_{M_{t}}\Hscr{}dV.
\end{equation*}
Then from $(3.4)$ and $(4.3)$, we have
\begin{eqnarray}
   \ddl{}{t}\alpha(t)
      &=& \int_{S(t)}H(\Hscr-H)d\mu -
          \int_{S(t)}\Hscr(\Hscr-H)d\mu \nonumber \\
      &=& -\int_{S(t)}(H-\Hscr)^{2}d\mu.
\end{eqnarray}
Therefore the function $\alpha(t)$ is non-increasing along the
flow. The integral
\begin{equation}\label{L1}
   \int_0^{\infty}\int_{S(t)}(H-\Hscr)^{2}d\mu dt
   = \alpha(0)-\alpha(\infty) <\infty,
\end{equation}
and we find that the integral $\int_{S(t)}(H-\Hscr)^{2}d\mu$ is
uniformly bounded.

We compute the $t-$derivative of this integral:
\begin{eqnarray*}
   \ddl{}{t}\int_{S(t)}(H-\Hscr)^{2}d\mu
      & = & \int_{S(t)}2(H-\Hscr)(H_t-\Hscr_t)-H(H-\Hscr)^3 d\mu \\
      & = & \int_{S(t)}-2|\nabla(H-\Hscr)|^2+2(H-\Hscr)^2(|A|^2-2)d\mu \\
      & + & \int_{S(t)}-H(H-\Hscr)^3 -2\Hscr_t(H-\Hscr)d\mu.
\end{eqnarray*}
Here we applied $(3.5)$.

The absolute value can be bounded by:
\begin{eqnarray*}
   \left| \ddl{}{t}\int_{S(t)}(H-\Hscr)^{2}d\mu\right|
       & \le &\sup_{S(t)}(|H(H-\Hscr)|+
              2|A|^2)\int_{S(t)}(H-\Hscr)^2 d\mu \\
       & + & 2\int_{S(t)}(|\nabla(H-\Hscr)|^2+
         |\inner{\nablabar{\Hscr}}{\nu}| |H-\Hscr|^2)d\mu.
\end{eqnarray*}
where we used
\begin{equation*}
   \left|\int_{S(t)}\Hscr_t (H-\Hscr) d\mu\right|
   \le \int_{S(t)}|\inner{\nablabar{\Hscr}}{\nu}| |H-\Hscr|^2)d\mu.
\end{equation*}
Recall that the {\ub}s for $H(\cdot,t)$ follows from Theorem 3.10
and the inequality $2|A|^2 \ge H^2$. Its gradient bound is standard
from the {\ub}s of $|A|^2$ and $|\nabla A|^2$. Now using the {\ub}s
on $H$, $\Hscr$, $|A|^2$, and their gradients, as well as Theorem 3.8,
we find that the term $|\ddl{}{t}\int_{S(t)}(H-\Hscr)^{2}d\mu|$ is
also uniformly bounded in $t$. Therefore it must tend to zero as
$t \to \infty$. This together with \eqref{L1} implies the $L^2$
bound:
\begin{equation*}
   \sup_{t \to \infty}\|H(\cdot,t) - \Hscr\|_{L^2}= 0.
\end{equation*}
This $L^2$-estimate in conjunction with {\ub} on $|\nabla(H-\Hscr)|$
allow us to apply the standard
interpolation argument (see for example, \cite[pp. 88--95]{Aub98})
to show the $L^{\infty}$ bound, i.e., \eqref{sup}.
\end{proof}

\begin{proof}[Proof of Theorem 3.1 (iii)]
Since all {\es}s stay in a compact region (Theorem 3.8), and we have
obtained {\ub}s for $|A|^2$ and its derivatives (Theorem 3.10),
we can employ the theorem of
Arzela-Ascoli to extract a subsequence of $S(t_i)$ converging to
a limiting surface, of {\mc} function
$\Hscr$ (Theorem 4.1).

The {\es}s $\{S(t)\}_{t \ge 0}$ remain as graphs over the {\rs},
therefore, by the estimates on the
{\hf} $u(\cdot,t)$ and its derivatives (Theorem 3.8, Proposition 3.9),
the limiting surface is also a graph over the {\rs}, hence embedded.
\end{proof}

The Corollary 1.5 is a direct application of the main theorem,
by taking constant functions as the prescribed {\mc}s. One shall
note that Theorem 1.4 does not directly apply to the case of {\vp}
{\mcf}, the situation we treated in \cite{HW09a}.


\begin{thebibliography}{GHW09}

\bibitem[And02]{And02}
Ben Andrews, \emph{Positively curved surfaces in the three-sphere}, Proc. of
  the ICM \textbf{2} (2002), 221--230.

\bibitem[Aub98]{Aub98}
Thierry Aubin, \emph{Some nonlinear problems in {R}iemannian geometry},
  Springer Monographs in Mathematics, Springer-Verlag, Berlin, 1998.

\bibitem[Bar84]{Bar84}
Robert Bartnik, \emph{Existence of maximal surfaces in asymptotically flat
  spacetimes}, Comm. Math. Phys. \textbf{94} (1984), no.~2, 155--175.

\bibitem[Ber72]{Ber72}
Lipman Bers, \emph{Uniformization, moduli, and {K}leinian groups}, Bull. London
  Math. Soc. \textbf{4} (1972), 257--300.

\bibitem[Bra78]{Bra78}
Kenneth~A. Brakke, \emph{The motion of a surface by its mean curvature},
  Mathematical Notes, vol.~20, Princeton University Press, Princeton, N.J.,
  1978.

\bibitem[CM04]{CM04}
Tobias~H. Colding and William~P. Minicozzi, II, \emph{Sharp estimates for mean
  curvature flow of graphs}, J. Reine Angew. Math. \textbf{574} (2004),
  187--195.

\bibitem[Eck03]{Eck03}
Klaus Ecker, \emph{Mean curvature flow of spacelike hypersurfaces near null
  initial data}, Comm. Anal. Geom. \textbf{11} (2003), no.~2, 181--205.

\bibitem[EH89]{EH89}
Klaus Ecker and Gerhard Huisken, \emph{Mean curvature evolution of entire
  graphs}, Ann. of Math. (2) \textbf{130} (1989), no.~3, 453--471.

\bibitem[EH91]{EH91}
\bysame, \emph{Parabolic methods for the construction of spacelike slices of
  prescribed mean curvature in cosmological spacetimes}, Comm. Math. Phys.
  \textbf{135} (1991), no.~3, 595--613.

\bibitem[Eps84]{Eps84}
Charles~L. Epstein, \emph{Envelopes of horospheres and weingarten surfaces in
  hyperbolic 3-spaces}, manuscript (1984).

\bibitem[Fri64]{Fri64}
Avner Friedman, \emph{Partial differential equations of parabolic type},
  Prentice-Hall Inc., Englewood Cliffs, N.J., 1964.

\bibitem[Ger83]{Ger83}
Claus Gerhardt, \emph{{$H$}-surfaces in {L}orentzian manifolds}, Comm. Math.
  Phys. \textbf{89} (1983), no.~4, 523--553.

\bibitem[GHW09]{GHW09}
Ren Guo, Zheng Huang, and Biao Wang, \emph{Quasi-fuchsian three-manifolds and
  metrics on {\TS}}, preprint, arXiv:0909.2426, (2009).

\bibitem[Ham82]{Ham82}
Richard Hamilton, \emph{Three-manifolds with positive ricci curvature}, J.
  Differential Geom. \textbf{17} (1982), 255--306.

\bibitem[Hop89]{Hop89}
Heinz Hopf, \emph{Differential geometry in the large}, Lecture Notes in
  Mathemtaics, vol. 1000, Springer-Verlag, Berlin, 1989.

\bibitem[Hui84]{Hui84}
Gerhard Huisken, \emph{Flow by mean curvature of convex surfaces into spheres},
  J. Differential Geom. \textbf{20} (1984), no.~1, 237--266.

\bibitem[Hui86]{Hui86}
\bysame, \emph{Contracting convex hypersurfaces in riemannian manifolds by
  their {\mc}}, Invent. Math. \textbf{84} (1986), 463--480.

\bibitem[HW09a]{HW09a}
Zheng Huang and Biao Wang, \emph{Geometric evolution equations and foliations
  on quasi-fuchsian three-manifolds}, preprint, arXiv:0907.2899, (2009).

\bibitem[HW09b]{HW09b}
\bysame, \emph{Mean curvature flows in almost fuchsian manifolds}, preprint,
  arXiv:1001.4217 (2009).

\bibitem[HY96]{HY96}
Gerhard Huisken and Shing-Tung Yau, \emph{Definition of center of mass for
  isolated physical systems and unique foliations by stable spheres with
  constant mean curvature}, Invent. Math. \textbf{124} (1996), 281--311.

\bibitem[KS07]{KS07}
Kirill Krasnov and Jean-Marc Schlenker, \emph{Minimal surfaces and particles in
  3-manifolds}, Geom. Dedicata \textbf{126} (2007), 187--254.

\bibitem[Lie96]{Lie96}
Gary~M. Lieberman, \emph{Second order parabolic differential equations}, World
  Scientific Publishing Co. Inc., River Edge, NJ, 1996.

\bibitem[Mar74]{Mar74}
Albert Marden, \emph{The geometry of finitely generated kleinian groups}, Ann.
  of Math. (2) \textbf{99} (1974), 383--462.

\bibitem[Mes07]{Mes07}
Geoffrey Mess, \emph{Lorentz spacetimes of constant curvature}, Geom. Dedicata
  \textbf{126} (2007), 3--45.

\bibitem[Rub05]{Rub05}
J.~Hyam Rubinstein, \emph{Minimal surfaces in geometric 3-manifolds}, Global
  theory of minimal surfaces, Clay Math. Proc., vol.~2, Amer. Math. Soc.,
  Providence, RI, 2005, pp.~725--746.

\bibitem[Sim68]{Sim68}
James Simons, \emph{Minimal varieties in riemannian manifolds}, Ann. of Math.
  (2) \textbf{88} (1968), 62--105.

\bibitem[SSY75]{SSY75}
Richard Schoen, Leon Simon, and Shing-Tung Yau, \emph{Curvature estimates for
  minimal hypersurfaces}, Acta Math. \textbf{134} (1975), no.~3-4, 275--288.

\bibitem[SU82]{SU82}
J.~Sacks and K.~Uhlenbeck, \emph{Minimal immersions of closed {R}iemann
  surfaces}, Trans. Amer. Math. Soc. \textbf{271} (1982), no.~2, 639--652.

\bibitem[SY79]{SY79}
Richard Schoen and Shing-Tung Yau, \emph{Existence of incompressible minimal
  surfaces and the topology of three-dimensional manifolds with nonnegative
  scalar curvature}, Ann. of Math. (2) \textbf{110} (1979), no.~1, 127--142.

\bibitem[Thu82]{Thu82}
William~P. Thurston, \emph{The geometry and topology of three-manifolds}, 1982,
  Princeton University Lecture Notes.

\bibitem[Uhl83]{Uhl83}
Karen~K. Uhlenbeck, \emph{Closed minimal surfaces in hyperbolic
  {$3$}-manifolds}, Seminar on minimal submanifolds, Ann. of Math. Stud., vol.
  103, Princeton Univ. Press, Princeton, NJ, 1983, pp.~147--168.

\bibitem[Unt03]{Unt03}
Philip Unterberger, \emph{Evolution of radial graphs in hyperbolic space by
  their mean curvature}, Comm. Anal. Geom. \textbf{11} (2003), no.~4, 675--695.

\bibitem[Wan08]{Wan08}
Biao Wang, \emph{Foliations for quasi-{F}uchsian $3$-manifolds}, Jour. of Diff.
  Geometry (2008), to appear.

\end{thebibliography}

\providecommand{\bysame}{\leavevmode\hbox to3em{\hrulefill}\thinspace}
\providecommand{\MR}{\relax\ifhmode\unskip\space\fi MR }
\providecommand{\MRhref}[2]{%
  \href{http://www.ams.org/mathscinet-getitem?mr=#1}{#2}
}
\providecommand{\href}[2]{#2}

\end{document}